\documentclass[11pt,oneside]{amsart}
\usepackage[english]{babel}
\usepackage{fullpage}
\usepackage{amssymb,pstricks,amscd,epsfig}
\usepackage{graphicx}

\usepackage{amsmath, amsthm, amssymb}
\usepackage{pdfsync}
\usepackage[latin1]{inputenc}
\usepackage{stmaryrd}
\usepackage{color}
\usepackage{setspace}
\usepackage[matrix,arrow,tips,curve]{xy}

\numberwithin{equation}{section}

\makeatletter
 \renewcommand\section{\@startsection {section}{1}{\z@}%
     {-4.5ex \@plus -1ex \@minus -.2ex}%
     {2.3ex \@plus.8ex}%
    {\centering\scshape}}

\usepackage{cite}

\newcommand{\Q}{\mathbb{Q}}

\renewcommand{\P}{\mathbb{P}}

\newcommand{\Shhom}{\mc{H}om}
\newcommand{\Shend}{\mc{E}nd}

\newcommand{\mc}{\mathcal}

\newcommand{\be}{\begin{equation}}
\newcommand{\ee}{\end{equation}}
\renewcommand{\phi}{\varphi}

\newcommand{\wt}{\widetilde}

\newtheorem{thm}{Theorem}[section]

\newtheorem{prop}[thm]{Proposition}
\newtheorem{lemma}[thm]{Lemma}
\newtheorem{cor}[thm]{Corollary}
\theoremstyle{definition}

\newtheorem{rem}[thm]{Remark}

\numberwithin{equation}{section}

\newcommand{\calJ}{\mathcal{J}}
\newcommand{\calY}{\mathcal{Y}}
\newcommand{\calC}{\mathcal{C}}
\newcommand{\calD}{\mathcal{D}}

\newcommand{\bP}{\mathbb{P}}

\DeclareMathOperator{\Sym}{Sym}

\DeclareMathOperator{\Prym}{Prym}

\DeclareMathOperator{\Fix}{Fix}

\DeclareMathOperator{\Sing}{Sing}

\DeclareMathOperator{\Spec}{Spec}

\title{The Euler number of hyper-K\"ahler manifolds of OG$10$ type}

\author[K. Hulek]{Klaus Hulek}
\address{Institut f\"ur Algebraische Geometrie, Leibniz Universit\"at Hannover,  30060 Hannover, Germany}
\email{hulek@math.uni-hannover.de}

 \author[R. Laza]{Radu Laza}
\address{Stony Brook University,  Stony Brook, NY 11794, USA}
\email{radu.laza@stonybrook.edu}

\author[G. Sacc\`a]{Giulia Sacc\`a}
\address{Columbia University, New York, NY 10027, USA}
\email{gs3032@columbia.edu}

\thanks{First author is grateful to DFG for partial support under grant Hu 337/7-1, the second author is supported in part by NSF grants DMS-1254812 and DMS-1802128, the third author is supported in part by NSF grant DMS-1801818. The last two authors wish to thank  Leibniz Universit\"at Hannover for the warm hospitality and good working conditions during their visits.
}

\begin{document}

\begin{abstract}
Using the \cite{LSV} construction, we give a simple proof for the fact that the Euler characteristic of a hyper-K\"ahler manifold of OG10 type is $\chi(\operatorname{OG10})=176,904$, a result previously established by Mozgovoy \cite{Mozgovoy}. \end{abstract}

\maketitle

\bibliographystyle{amsalpha}

\section{Introduction}
Algebraic manifolds with trivial canonical bundle, or more generally Ricci-flat compact K\"ahler manifolds, are an important class of manifolds and play a special role in the classification of algebraic varieties. By the famous decomposition theorem of 
Beauville and Bogomolov \cite{B-K0}, Ricci flat compact K\"ahler manifolds are, up to finite cover, products of tori, Calabi-Yau manifolds (CY) and  hyper-K\"ahler manifolds (HK), the latter also known as irreducible holomorphic symplectic manifolds (IHSM).  
So far only very few examples of hyper-K\"ahler manifolds are known: these are two infinite series (with one case in each even dimension) namely manifolds which are deformation equivalent to Hilbert schemes of points on $K3$ surfaces and  so-called
generalized Kummer varieties, together with two sporadic examples in dimension $6$ and $10$ respectively, due to O'Grady (denoted OG6 and OG10 below). 
It is a basic question to understand the topology (e.g. the Betti numbers) of these manifolds.  The two infinite series are closely related to symmetric powers of $K3$ surfaces and abelian surfaces respectively, leading to a full description of their cohomology rings
(e.g. \cite{Go,GS}). 
The topology of the six dimensional O'Grady example was determined by the third author and her collaborators \cite{MRS}. 
The purpose of this note is to give a simple proof for the computation of the topological Euler characteristic for OG10, a result 
 first established in the thesis of S. Mozgovoy \cite{Mozgovoy}.

\begin{thm}\label{mainthm}
The  Euler characteristic for a hyper-K\"ahler manifold of OG10 type is $176,904$.
\end{thm}

An obvious natural question is to determine  the Betti (and Hodge) numbers of   hyper-K\"ahler manifolds $Z$ of OG10 type. This will be discussed elsewhere\footnote{(Note added in proof) This was now settled in \cite{dCRS}. Subsequently, further information on the cohomology of    hyper-K\"ahler manifolds of OG10 type was obtained in \cite{GKLR}.}. 
For now, we only note the following known (mostly general) restrictions on the Betti numbers. Clearly, by definition, $b_1(Z)=0$ for all hyper-K\"ahler manifolds. Also, the second Betti number has been computed  for all known examples of HK. In particular, the case  
of OG10 was done by Rapagnetta  \cite{Rapagnetta}, who showed that $b_2(Z)=24$. 
By Verbitsky \cite{Verbitsky} (see also \cite{LL}), it is also known that for any HK $Z$ of dimension $2n$, and any  $k\in\{2,\dots,n\}$,  the cup product defines a natural inclusion $\Sym^k H^2(Z,\mathbb Z)\hookrightarrow H^{2k}(Z,\mathbb Z)$, and thus $b_{2k}\ge {b_2+k-1 \choose k}$. 
Salamon \cite{Salamon} 
proved that
$$2\sum_{l=1}^{2n} (-1)^l(3l^2-n)b_{2n-l}=nb_{2n}.$$
Together with the knowledge of the Euler number, these relations give some strong restrictions on the Betti numbers, but not sufficient to determine them for  hyper-K\"ahler manifolds of OG10 type.

Our argument for the computation of $\chi(Z)$ is an adaptation of Beauville's method for counting curves on a K3 surface \cite{Beau} which extends the standard computation of the Euler number for $K3$ surfaces by counting the number of singular fibers in an elliptic $K3$ surface. 
Similarly, our starting point is the construction of \cite{LSV} for hyper-K\"ahler manifolds $Z$ of OG10 type as Lagrangian fibrations $Z=\overline \calJ \to B$ associated to a cubic fourfold $X \subset \mathbb P^5$. Namely, the general fiber $ \calJ_b$ is the intermediate Jacobian of the cubic threefold $Y_b=X \cap  H_b$ where $H_b$ is the hyperplane corresponding to a point $b \in B:=({{\mathbb P}^5)}^{\vee}$. This leads, as originally observed in \cite{Donagi-Markman}, to an open Lagrangian fibration $\calJ/U$ over the smooth locus $U =B\setminus X^\vee$ (with $\calJ=\overline \calJ_{\mid U}$). 
On the other hand, by a result of Mumford \cite{MumPrym} the intermediate Jacobian  $\calJ_b$ is isomorphic to a Prym variety. Using this description, in \cite{LSV}, the compactification $Z=\overline \calJ$ was then constructed, \'etale locally, as a relative compactified Prym over $B$. Returning to the proof of Theorem \ref{mainthm}, 
by standard arguments, recalled in  Section \ref{secprelim},   it will be enough to consider only the fibers with   $\chi(\overline \calJ_b)\neq 0$ (Cor. \ref{cor_euler}), which are only finitely many. Now, by construction, the discriminant in $B$ of $\overline \calJ/B$ is the dual variety $X^\vee$, which is naturally stratified in terms of singularities of the tangent hyperplanes to $X$ (i.e. the singularities of the cubic threefold $X\cap H$). For a point in the discriminant, the associated limit compactified intermediate Jacobian has non-zero Euler characteristic only if it has no abelian factor, which in turn is equivalent, under the assumption that $X$ is generic, to saying that $H$ is a $5$-tangent hyperplane to $X$. As in the case of very general elliptically fibered $K3$s, the contribution (of the $5$-tangent hyperplane sections) to the Euler characteristic is $1$ (Cor. \ref{chione}). Thus the computation of the Euler characteristic reduces to the enumerative question of counting the number of $5$-tangent hyperplanes to a general cubic fourfold. These type of questions (and much more) are answered by the theory of Thom polynomials of singularities (e.g. Rim\'{a}nyi \cite{Rim}, Kazarian \cite{Kazarian}).

At this point we would like to compare our approach to that of  Mozgovoy  \cite{Mozgovoy}. His starting point goes back to O'Grady \cite{OG10} (see also \cite{Rapagnetta}) who first constructed hyper-K\"ahler manifolds of OG10 type as symplectic resolutions $\widetilde M$
of certain moduli spaces $M$ of sheaves on $K3$ surfaces $S$. Namely, applying O'Grady's construction to the case of 
polarized $K3$ surfaces $(S,H)$  of degree $2$, one notes that the linear system $|2H|$ has dimension $5$ and then by associating to each sheaf its Fitting support, one obtains a 
fibration $\widetilde M \to |2H|$ whose fibers are themselves moduli spaces of sheaves on (possibly singular) curves $C \in |2H|$. Similar to the argument that we use here, Mozgovoy then uses this fibration and the additivity of the Euler number to compute $\chi(\widetilde M)$. 
However, this is technically somewhat involved 
as the curve $C$ can be singular, reducible and even non-reduced. Consequently, one needs to keep track of  the Euler numbers for various special fibers. In contrast, in our situation there is only one type of relevant special fiber, which comes with multiplicity $1$ and Euler characteristic $1$. Furthermore, one can view Mozgovoy's computation as a degeneration of our computation.  This is due to the fact that cubic fourfolds degenerate to the secant variety of the Veronese surface in $\mathbb P^5$; the limit mixed Hodge structure associated to such a $1$-parameter degeneration is pure, and, in fact, can be naturally identified with the Hodge structure of a degree $2$ $K3$ surface $S$ (see \cite{La10}). Keeping track of the associated \cite{LSV} fibration as the cubic fourfold $X$ degenerates to the secant variety of the Veronese surface, one recovers the original O'Grady construction associated to $S$ as described above (see \cite[\S5.3]{KLSV} for details). Finally, in the limit, the $5$-tangent hyperplanes that we consider lead to the special curves that enter into Mozgovoy's computation. In other words, the locus $\mathfrak L$ of Lagrangian fibered OG10 manifolds constructed by \cite{LSV} is $20$-dimensional, giving a Noether--Lefschetz divisor in the moduli space $\mathfrak M$ of polarized OG10 hyper-K\"ahler manifolds, while the locus $\mathfrak L'$ obtained by O'Grady's method starting with a degree $2$ $K3$ surface $S$ is $19$-dimensional. The argument sketched above says that $\mathfrak L'\subset \mathfrak L$ (and, in fact, a divisor), showing that indeed one can regard Mozgovoy's computation as a limit of ours.

\section{The Prym construction of OG$10$ and enumerative geometry}\label{secprelim}
\subsection{Preliminaries}
We recall that the Euler characteristic for algebraic varieties satisfies  $\chi(Z)=\chi_c(Z)(:= \sum (-1)^i \dim_\Q H^i_c(Z, \Q))$ (e.g. \cite[p. 141]{fulton}). Consequently, the Euler characteristic is additive with respect to  open and closed embeddings, i.e. for $W\subset Z$ closed, and $U=Z\setminus W$, we have $\chi(Z)=\chi(W)+\chi(U)$. Furthermore,  the Euler characteristic is multiplicative for smooth proper  fibrations of algebraic varieties. In particular, in our set-up: fibrations in complex tori - the Euler characteristic for the smooth part is $0$ (e.g. $\chi(\calJ)=0$, where $\calJ/U$ is as in the introduction). Thus, it remains to consider the behavior of the fibration $\overline \calJ/B$ over the singular part. In fact, by considering a Whitney stratification, one can show that only the fibers with non-zero Euler characteristic are relevant for the computation. Moreover, it turns out that there is only a finite number of them. 
Specifically, the following holds.

\begin{prop} \label{eulerprop} 
Let $Z^\circ \to B^\circ$ be a proper morphism of complex algebraic varieties
such that  $\chi(Z_b)=0$  for all $b \in B^\circ$, then $\chi (Z)=0$.
\end{prop}
\begin{proof}
This is a particular case of \cite[Prop. 2.4]{CMS} which gives a general ``multiplicative'' formula for the Euler characteristic for a proper map of algebraic varieties (in terms of a Whitney stratification). 
\end{proof}

From this, we conclude:
\begin{cor}\label{cor_euler}
Let $Z\to B$ be a proper surjective morphism, and $\Sigma_0$ be a finite set  such that $\chi(Z_b)=0$ for $b\in B\setminus \Sigma_0$. Then $\chi(Z)=\sum_{b\in \Sigma_0} \chi(Z_b)$.  \qed
\end{cor}

As already mentioned, we will apply this result to the Lagrangian fibration $Z=\overline \calJ/B$ constructed in \cite{LSV} as a model for OG10 HK manifolds (we note that the locus of Lagrangian fibered OG$10$ is a codimension $1$ locus in moduli). Below, we review this construction and discuss the relevant stratification of the discriminant. 

\subsection{The \cite{LSV} construction of OG10}Let $X \subset \P^5$ be a general cubic fourfold, and let $B:=(\P^5)^\vee$ 
be the projective space parameterizing its hyperplane sections. We denote by $U \subset ({\P^5})^\vee$ the open locus parameterizing smooth hyperplane sections. A hyperplane section $Y_b=X\cap H_b$ for $b\in U$ is a smooth cubic threefold, whose associated intermediate Jacobian $J(Y_b)$  is a principally polarized abelian variety of dimension $5$ (cf. \cite{Clemens-Griffiths}). Considering the family of such intermediate Jacobians leads to a morphism of quasi-projective varieties $\pi_U:\calJ\to U$, and furthermore $\calJ$ carries a holomorphic symplectic form, with respect to which $\calJ/U$ is a Lagrangian fibration (see \cite[\S 8.5.2]{Donagi-Markman}). The content of \cite{LSV} is the construction of a smooth compactification $\overline \calJ/B$ of $\calJ/U$ such that the holomorphic form extends and remains non-degenerate.

\begin{thm}[{Laza--Sacc\`a--Voisin \cite{LSV}}]
Let $X$ be a general cubic fourfold. There exists a smooth projective compactification $Z=\overline \calJ$ of $\calJ$,  which is a hyper--K\"ahler manifold and such that $\pi_U$ extends to a Lagrangian fibration $\pi:\overline \calJ \to B$. 
Moreover, $Z=\overline \calJ$ is of OG10 type. 
\end{thm}

As discussed above, in order to prove Theorem \ref{mainthm}, we need to understand the singular fibers of $\overline \calJ/B$ and their Euler characteristic. This is closely related to the  study of degenerations of intermediate Jacobians (see esp. \cite{CML,CGHL}).  The main tool for studying degenerations of intermediate Jacobians is Mumford's description of the intermediate Jacobian as a Prym variety. Specifically, if $Y$ is a smooth cubic threefold, and $\ell\subset Y$ is a general line, then the projection from $\ell$ realizes $\widetilde Y:=\mathrm{Bl}_{\ell} Y$ as a conic bundle over $\bP^2$. The discriminant of $\widetilde Y\to \bP^2$ is a quintic curve $C$, and 
furthermore $\widetilde Y\to \bP^2$ naturally determines an \'etale double cover $D\to C$. Mumford's theorem then says that $J(Y)\cong \Prym(D,C)$. Based on earlier results of Beauville, Casalaina-Martin and Laza \cite{CML} have noticed that the Prym construction also works well in the singular case (for $Y$ mildly singular), as long as one makes a careful choice of the line $\ell$. Furthermore, if $Y$ is any hyperplane section of a (Hodge) general cubic fourfold $X$, then $\ell$ can be chosen to be a ``very good line'' (see \cite[Def. 2.9]{LSV}). In short, the relevant statement for us is the following:

\begin{prop}[{\cite[Prop. 2.3]{LSV}, \cite{CML}}]  \label{verygeneralline}
Let $X$ be a general cubic fourfold. Then for any hyperplane section $Y=X\cap H$ there exists a line $\ell \subset Y$ such that 
\begin{enumerate}
\item The double cover $f: D \to C$ associated to the conic bundle $\widetilde Y\to \bP^2$ is  \'etale  and both curves $D$ and $C$ are irreducible;
\item  The singularities of $Y$ and those of $C$ are in one-to-one correspondence, including the analytic type (i.e. there is a bijection  $\phi:\Sing(Y)\to \Sing(C)$, and the germ $(Y,y)$ is a double suspension of  $(C,\phi(y))$).  
\end{enumerate}
\end{prop}

\begin{rem}\label{rem_versal}
A key fact about the hyperplane sections of a {\it general} cubic fourfold $X$ is that the linear system of hyperplane sections of $X$ gives a simultaneous versal deformation of the singularities of any hyperplane section $Y$ of $X$ (see \cite[Prop. 3.6]{LSV}; see also \cite[Sect. 3]{CGHL} for a related discussion). The same is  true for the associated curves $C$. More precisely, given $Y$ and a choice of very good line $\ell$, one gets a (possibly singular) quintic $C$. A small embedded deformation of $(Y,\ell)$ (in $X$) determines a family of quintics, which versally deform the singularities of $C$.  In particular, this bounds the Milnor number of the singularities of $C$, and by standard singularity theory, it follows that only the types 
$A_1,\dots, A_5, D_4, D_5$ can occur. Finally,  property (2) above says that $C$ has the same number and type of $A_i$ and $D_k$ singularities as $Y$. 
\end{rem}

Returning to the \cite{LSV} construction of $\overline \calJ/B$, we note that locally, in the \'etale or analytic topology, $\overline \calJ/B$ is the relative compactified Prym associated to a family of curves $(\calD,\calC)$ which is obtained via projection from a (local) family of good lines on the universal family of hyperplane sections $\calY/B$ of $X$. By \cite[Prop. 5.1 and Thm. 5.7 ]{LSV} the fiber of $\pi:\overline \calJ\to B$ over  $b \in B$ is the compactified Prym variety of a double cover $D \to C$ of irreducible locally planar curves, i.e.
\begin{equation}\label{defjb}
\overline \calJ_b=\overline \Prym(D,C)=:\bar P_{D/C}.
\end{equation}
(We note that since both $C$ and $D$ are irreducible with planar singularities, there is no ambiguity in the definition of the compactified Prym). We recall that the compactified Prym is a natural adaptation (in the double cover set-up) of the compactified Jacobian. We refer to  \cite[\S 4 ]{LSV}  and to Section \ref{section prym}  for the relevant notation, definitions and first properties of compactified Prym varieties. For the moment, we recall that the compactified Prym variety has an abelian variety factor, namely
the ``compact part'' of the generalized Prym variety of $D$ over $C$ (see Section \ref{section prym} for the relevant definitions). As discussed in Section \ref{section prym} below (based on ideas from \cite{Beau}), the relevant case for us is when  this abelian factor vanishes. Under the \'etale assumption, this is equivalent to saying that the genus of the normalization of $C$ is $1$, a case that is described geometrically below.

\begin{lemma} \label{only5A1}
Let $C$ be a plane quintic with a combination of $A_l$ and $D_k$ singularities, denoted $\sum_l m_l A_l+ \sum_k n_k D_k$. Let $g=g(C)$ be the geometric genus, and $\mu_{tot}(C)=\sum_l l\cdot m_l +\sum_k k\cdot n_k$ be the total Milnor number. Assume that $C$ is irreducible and $\mu_{tot}(C)\le 5$. Then $g \ge 1$, and $g=1$ iff $\sum_l m_l A_l+ \sum_k n_k D_k=5A_1$.
\end{lemma}
\begin{proof} The arithmetic genus of a quintic curve $C$ is $6$. Since $C$ is irreducible, each relevant singularity gives a genus drop of $\delta$ according to the following table 
\[
\begin{tabular}{ |c|c|c|c|c| } 
 \hline
Singularity & $A_{2l-1}$ & $A_{2l}$ & $D_4$ & $D_5$ \\
 \hline
$\delta$ & l & l &  3 & 3\\
 \hline
 \end{tabular}
\]
Assuming $C$ is as in the lemma, we get 
\[
p_a(C)-p_g(C)=\sum _l m_l \delta(A_l) + \sum_k n_k \delta(D_k) \le \sum _l m_l \cdot  l+ \sum_k n_k \cdot k \le 5.
\]
We thus obtain $p_g(C)\ge 1$, and the equality holds if and only if $C$ has $5A_1$ singularities. 
\end{proof}

\subsection{Stratification of the dual variety $X^\vee$} Motivated by Lemma \ref{only5A1}, we will need to count the hyperplane sections $Y$ of $X$ that lead to $5$ nodal plane quintics (via the projection from a very good line). In view of Proposition \ref{verygeneralline}, this is equivalent to counting the $5$-tangent hyperplanes to a general cubic fourfold.  This is part of a more general question regarding the structure of the dual variety $X^\vee$ that we briefly review below. 

Let $X\subset \bP^5$ be a general cubic fourfold. The dual variety $X^{\vee}$ is naturally stratified in terms of the singularities of the associated hyperplane section $Y=X\cap H_b$ (for $b\in X^\vee\subset B$). More precisely, $Y$ will have some combination of $AD$ singularities $R=\sum_l m_l A_l+ \sum_k n_k D_k$ with $m_l, n_k \ge 0$, i.e. $Y$ has exactly $m_l$ singular points of type $A_l$ and $n_k$ singular points of type $D_k$. Prescribing a combination of singularities $R$ will define a stratum $\Sigma_R$ of $X^\vee$. In our set-up (i.e. $X$ a general cubic), we know (cf. Rem. \ref{rem_versal}) that at worst $A_5$ and $D_5$ occur and furthermore the codimension of the stratum $\Sigma_R$ associated to $R=\sum_l m_l A_l+ \sum n_k D_k$ is
\begin{equation}\label{eq_mutot}
\mu_{tot}(R)=\sum_l l \cdot m_l+\sum_k k\cdot n_k\le 5.
\end{equation}
The versality property of Remark \ref{rem_versal} easily allows one to determine the incidence of various strata (e.g. $\Sigma_{D_4}\subset \overline \Sigma_{3A_1}$); we refer the interested reader to \cite{CML2} for further discussion of the local structure of the strata $\Sigma_R$, and their geometric relevance. What is relevant here is to note that each $\overline\Sigma_R$ is a projective variety in $(\bP^5)^\vee$ and thus has a degree $d_R$ and (expected) codimension $\mu_{tot}(R)$ (e.g. $d_{A_1}=\deg(X^\vee)=3\cdot 2^4$). The computation of the degree $d_R$ is a classical question in enumerative geometry and singularity theory. The theory of Thom polynomials (Rim\'{a}nyi \cite{Rim}, Kazarian \cite{Kazarian}) gives an effective method of computing the various $d_R$ as long as the simultaneous versal property (cf. Rem. \ref{rem_versal}) holds (in particular, the expected codimension $\mu_{tot}(R)$ is the actual codimension). For the low dimensional cases and small $\mu_{tot}(R)$, Kazarian \cite{Kazarian} gave explicit formulae. In particular, all that is needed for our purposes is the degree $\deg(\overline \Sigma_{5A_1})$, or equivalently the number of $5$-tangent hyperplanes to a general cubic fourfold $X$. 

\begin{thm}[{Kazarian \cite[Sect. 10]{Kazarian}}]\label{theoremK}
Let $X$ be a general cubic fourfold. Then there are exactly $176,904$ hyperplanes $H$ which are $5$ tangent to $X$. 
\end{thm}
\begin{proof}
The specific formula relevant to us is listed in \cite[p. 6--7]{K2} (see ``enum[4,5]'' in loc. cit.). For the reader's convenience, we reproduce the formula for the number of $5$-tangent hyperplanes to a general degree $d\ge 3$ hypersuface in $\bP^5$:
\begin{eqnarray*}
\scriptstyle
m_{5A_1}(d)&\scriptstyle=&\scriptstyle\frac{1}{120}(d-2)d(d^{23}-18 d^{22}+154 d^{21}-832 d^{20}+3,181 d^{19}-9,332 d^{18}+23,306 d^{17}-56,258 d^{16}+137,704 d^{15}-315,702 d^{14}\\&&\scriptstyle +632,037 d^{13}-1,167,746 d^{12}+2,276,543 d^{11}- 4,606,484 d^{10}+8,183,892 d^9-12,182,630 d^8+19,262,625 d^7\\
&&\scriptstyle- 37,322,080 d^6+63,347,155 d^5-72,821,310 d^4+73,475,394 d^3-156,527,928 d^2+284,455,368d-193,415,040)
\end{eqnarray*}
Setting $d=3$, we get $m_{5A_1}(3)=176,904$ as claimed.
\end{proof}

\begin{rem}
For comparison, we recall the situation for lower dimensional cubics. It is a standard fact  that a cubic surface has $45$ tritangent hyperplanes. For a general cubic threefold $Y$,  there are $495(=2^4(2^5-1)-1)$ $4$-tangent hyperplanes to $Y$. This can be obtained as a special case of Kazarian's results, or alternatively (and more geometrically), as the number of non-trivial odd theta characteristic for the intermediate Jacobian $J(Y)$. The latter claim follows by using the Prym description $J(Y)=\Prym(D,C)$ as above, and relating the $4$-tangent hyperplanes to $Y$ to a certain configuration of conics relative to the quintic $C$ (which was studied in  \cite{White}). We note that $176,904=3^5(3^6-1)$, which indicates a relationship to the group of $3$-torsion points on an abelian variety, but we are not aware of a direct geometric link. 
\end{rem}

\section{The Euler characteristic of compactified Prym varieties} \label{section prym}
The \cite{LSV} model $Z=\overline \calJ/B$ of OG10 HK manifolds can be understood by means of relative compactified Prym varieties associated to double covers of plane quintics. Here, after a brief review of the compactified Prym varieties, we discuss the  Euler characteristic of compactified Pryms. The main results (Prop. \ref{chizero} and Cor. \ref{chione}) are analogous to results of Beauville \cite{Beau} for Jacobians. 

\subsection{Compactified Prym varieties}
Let $f: D \to C$ be an \'etale double cover of irreducible locally planar curves and let $\iota: D \to D$ be the fixed point free involution associated to the covering. We denote by $J_D^d$ the degree $d$ generalized Jacobian of $D$ and by $\bar J^d_D$  its degree $d$ compactified Jacobian, parameterizing locally free and torsion free sheaves of rank $1$ and degree $d$, respectively. We recall that since $D$ is irreducible with locally planar singularities, $\bar J^d_D$ is irreducible, and its smooth locus is precisely $J_D^d$ (e.g. \cite{Rego}). We denote by $J_D$ and $\bar J_D$ the degree $0$ generalized and compactified Jacobians. Notice, however, that because $D$ is irreducible, $\bar J^d_D$  is independent of the degree.

In \cite[\S 4]{LSV} (cf. also \cite[\S 3]{ASF}), the compactified Prym variety is defined  as the identity component of the fixed locus of the involution 
\[
\begin{aligned}
-\iota^*: \bar J_D &\longrightarrow \bar J_D\\
F & \longmapsto (\iota^* F)^\vee:=\Shhom_{{\mc O_D}}(\iota^* F, \mc O_D)
\end{aligned}
\]
acting on the degree zero compactified Jacobian of $D$. In formulae:
\[
\bar P_{D/C}:=\Fix^\circ (-\iota^*) \subset \bar J_D.
\]
We refer the reader to \cite[\S 4]{LSV} for more details on this construction.
Let $g=p_a(C)$ be the arithmetic genus of $C$. By \cite[Prop. 4.10 and Cor. 4.16]{LSV}, the compactified Prym variety $\bar P_{D/C}$ is an irreducible projective variety of dimension $g-1$. The open dense subset $P_{D/C}:=\bar P_{D/C} \cap J_D$ parameterizing line bundles, also called {\it the generalized Prym variety}, can be described in the following way.
 Let $\wt D$ and $\wt C$ be the normalization of the curves $D$ and $C$ respectively, and denote by $\wt g$ the genus of $\wt C$ so that $\wt g= g-\delta$  for some $\delta \ge 0$. There is a natural \'etale double cover $\wt D \to \wt C$ and $P_{D/C}$ fits in the exact sequence of algebraic groups
 \be \label{G}
0 \to G \to P_{D/C} \to P_{\wt D /\wt C} \to 0,
 \ee
where $G$ is an affine group of dimension $\delta$ (a product of additive and multiplicative groups) and the Prym variety $P_{\wt D /\wt C} $ is a principally polarized abelian variety of dimension $\wt g-1$. More precisely, $G$ is isomorphic to $\mc O^*_{\wt C} \slash \mc O^*_C=\oplus _{x \in \operatorname{Sing}(C)}(\mc O^*_{\wt C} \slash \mc O^*_C)_x$ and can be viewed inside $\mc O^*_{\wt D} \slash \mc O^*_D\cong G \times G$ with the anti--diagonal embedding. 

It is well known that $J_D$ acts on $\bar J_D$ by tensorization. It is shown in \cite[Lem 2.1]{Beau} that the stabilizer of every point can be described in the following way. First recall that if $F$ is a rank one torsion free sheaf, then there is a partial normalization $n': D' \to D$ and a torsion free sheaf $F'$ on $D'$, such that $\Shend_{{\mc O_D}'}(F')\cong \mc O_{D'}$ and $F=n'_* F'$. The curve $D'$ is uniquely determined by the condition $D'=\Spec_{\mc O_{D}} \Shend_{\mc O_D} (F)$. 
Moreover, if $\deg F=0$ and $\delta'=p_a(D)-p_a(D')$, then $\deg F'=-\delta'$. Finally, by \cite[Lem 3.1]{Beau}  the morphism
\[
n_*': \bar J_{D'}^{-\delta'} \to \bar J_D
\]
is an embedding. By \cite[Lem. 2.1]{Beau} the stabilizer of $F$ in $J_D$ is precisely the kernel of the pullback $({n'})^*: J_D \to J_{D'}$. By restriction, $P_{D/C}$ acts on $\bar J_{D}$ preserving $\Fix(-\iota^*)$, so there is an action of $P_{D/C}$ on $\bar P_{D/C}$. Let $F$ be a point in $\bar P_{D/C}$. An isomorphism $F \cong (\iota^* F)^\vee$, determines isomorphisms $\Shend_{\mc O_{D,x}} F_x \cong \Shend_{\mc O_{D,\iota(x)}} F_{\iota(x)}$ for every $x \in D$  inducing an involution $\iota': D' \to D'$ naturally lifting $\iota$. The curve $C'=D' \slash \iota'$ is  a partial normalization of $C$ and there is a corresponding pullback map between  Prym varieties
\[
({n'})^*: P_{D/C} \to P_{D' \slash C'}.
\]
The kernel of this morphism is naturally identified with $\mc O_{ C'}^* \slash \mc O_{C}^*$ and is precisely the stabilizer in $P_{D \slash C}$ of $F$. 

\subsection{The Euler characteristic of compactified Pryms} This proof of the following is an adaptation to compactified Pryms of the analogous statement by Beauville \cite[Prop. 2.2]{Beau}.

\begin{prop} \label{chizero}
Let $f: D \to C$  and $\wt g$ be as above. If $\wt g \ge 2$ then $\chi(\bar P_{D \slash C})=0$.
\end{prop}
\begin{proof} 
It is enough to show that for any  integer $n\ge 2$ there is a free action  of a group of order $n$ on $\bar P_{D \slash C}$. Indeed, this implies that $\chi(\bar P_{D \slash C})$ is divisible by $n$ for every integer $n \ge 2$, and thus $\chi(\bar P_{D \slash C})=0$. Consider the sequence (\ref{G}). Since $G$ is a divisible (hence injective) abelian group, this sequence is split (as a sequence of abelian groups). It follows that as long as $ P_{\wt D /\wt C}$ is an abelian variety of dimension $\ge 1$ (i.e. as long as $\wt g \ge 2$), we can find a group $K_n$ of order $n$ in $P_{D\slash C}$ which maps injectively to $ P_{\wt D /\wt C}$.  Since by the above discussion the stabilizer of any point of $\bar P_{D\slash C}$ is contained in the kernel of $P_{D/C} \to P_{\wt D \slash \wt C}$, it follows that $K_n$ acts freely on $\bar P_{D\slash C}$. This completes the proof.
\end{proof}

We are now left with computing the Euler characteristic of the compactified Prym variety of an \'etale double cover of irreducible curves of geometric genus $1$. In view of Lemma \ref{only5A1} we only need to focus on nodal curves, so for the rest of this section we make the following assumption
\[
(*) \,\, C \text{ is a nodal curve.}
\]

We recalled earlier that for every partial normalization $n': D' \to D$ there is a natural closed embedding
\[
 n_*'(\bar J_{D'}^{-\delta'}) \subset \bar J_D.
\]
Notice the shift  by $-\delta'=-(p_a(D)-p_a(D'))$ in the degree. We wish to describe the intersection of $\bar P_{D \slash C}$ with each $ n_*'(\bar J_{D'}^{-\delta'}) $. We will do so expressing this intersection in terms of a ``twisted'' Prym. First, let us recall a few facts about relative duality applied to the finite morphism $n': D' \to D$.  Since $n'$ is a finite morphism, it admits a relative dualizing sheaf which we denote by $\omega_{n'}$.  By relative duality
\[
\Shhom_{\mc O_D}(n'_*F', \mc O_D)=n'_* \Shhom_{{\mc O_D}'}(F', \omega_{n'}).
\]
Since $D$ and $D'$ are nodal curves, their dualizing sheaves are locally free and   $\omega_{D'}=\omega_{n'} \otimes ({n'})^* \omega_D$ is a line bundle on $D'$ of degree $\deg \omega_{n'}=-2\delta'$.
There is a commutative diagram
\[
\xymatrix{
\bar J_{D'}^{-\delta'} \ar[d]_{n'_*} \ar[r] & \bar J_{D'}^{-\delta'} \ar[d]_{n'_*} \\
\bar J_D \ar[r] & \bar J_{D}
} \quad 
\xymatrix{
F' \ar@{|->}[d] \ar@{|->}[r] & \Shhom_{{\mc O_D}'}(F', \omega_{n'}) \ar@{|->}[d] \\
n'_* F'  \ar@{|->}[r] & n'_* \Shhom_{{\mc O_D}'}(F', \omega_{n'}).
} 
\]

\begin{prop} Let $D \to C$ be an \'etale double cover of nodal and irreducible curves and let $D' \to D$ be a partial normalization of $D$. Then
$\bar P_{D \slash C} \cap  n_*'(\bar J_{D'}^{-\delta'}) \neq \emptyset$ if and only if there is an involution  $\iota': D' \to D'$   lifting $\iota$. If this is the case, then
\[
\bar P_{D \slash C} \cap  n_*'(\bar J_{D'}^{-\delta'}) \cong \bar P_{D' \slash C'}
\]
where as above $C'=D' \slash \iota'$ is a partial normalization of $C$.
\end{prop}
\begin{proof} We  show that $\bar P_{D \slash C} \cap  n_*'(\bar J_{D'}^{-\delta'})$ is the image under $n'_*$ of a ``twisted'' Prym variety sitting in $\bar J_{D'}^{-\delta'}$. Let $F$ be a point in $\Fix(-\iota^*) \cap  n_*'(\bar J_{D'}^{-\delta'})$ and let $F'$ be such that $F=n_*' F'$. As observed above, this ensures that there is an involution $\iota': D' \to D'$, which lifts $\iota$. By uniqueness of the relative dualizing sheaf we see that $({\iota'})^* \omega_{n'} \cong \omega_{n'}$ and hence
\[
(\iota^* F)^\vee= \iota^* \Shhom_{\mc O_D}(n'_* F, \mc O_D)\cong n'_* \Shhom_{{\mc O_D}'}\left(({\iota'})^*F', \omega_{n'}\right).
\]
Here, we have used: that $\iota^*$ and $( \cdot )^\vee$ commute, since they commute on the dense open subset parametrizing locally free sheaves (similarly, $(\iota')^*$ and $ \Shhom_{{\mc O_D}'}\left( \cdot, \omega_{n'}\right)$ commute); that for any sheaf $\mc E$ on $D$, $\iota^* \mc E \cong \iota_* \mc E$ (and similarly for $\iota'$); and duality for finite morphisms (cf. Prop. 4.25 and Lem. 4.26 of \cite{Liu}).
It follows that, if $F \cong (\iota^* F)^\vee$, then $F'$ is a fixed point of the involution
\[
\begin{aligned}
\tau_{D'}:\bar J_{D'}^{-\delta'} & \longrightarrow \bar J_{D'}^{-\delta'} \\
F' & \longmapsto \Shhom_{{\mc O_D}'}\left(({\iota'})^*F', \omega_{n'}\right).
\end{aligned}
\]
This implies that
\be \label{fixed locus}
\Fix(-\iota^*) \cap  n_*'(\bar J_{D'}^{-\delta'})=\Fix(\tau).
\ee

Since $(\iota')^* \omega_{n'} \cong \omega_{n'}$ it is not hard to see (e.g. \cite[p. 329]{MumPrym}) that there exists a degree $\delta'$ line bundle $L$ on $D'$ such that
\[
\omega_{n'} \cong L^\vee \otimes (\iota')^* L^\vee.
\]
This shows that under the isomorphism $\bar J_{D'}^{-\delta'} \to \bar J_{D'}$ defined by tensoring with $L$, we have an isomorphism
\[
\Fix\left(-(\iota')^*\right) \cong  \Fix(\tau).
\]
Now by \cite[Cor. 4.16]{LSV}  both $\Fix(-\iota^*)$ and $\Fix(-(\iota')^*) \cong \Fix(\tau)$ have exactly four irreducible connected components which are isomorphic to each other. The compactified Prym variety is the one containing the identity, and the isomorphism of any component $Z$ with the Prym is defined by tensorization with a line bundle belonging to $Z$  (cf. \cite[(4.8)]{LSV}). This shows that this isomorphism preserves the local type of sheaves and hence that every component has the same strata appearing. Now look at (\ref{fixed locus}). The right hand side has $4$ connected components and hence so has the left hand side. By the discussion above, if one component intersects $n_*'(\bar J_{D'}^{-\delta'})$ then so do all the others. In particular, each component of $\Fix(-\iota^*)$ intersects $n_*'(\bar J_{D'}^{-\delta'})$ in a connected closed subset which has to be isomorphic to $ \bar P_{D' \slash C'}$.
\end{proof}

\begin{rem}
Without assuming that $C$ (and $D$) are nodal, the same conclusion holds true for any stratum corresponding to a partial normalization $D'$ that is also locally planar.  
\end{rem}

\begin{cor} \label{chione} If $C$ is a nodal curve of geometric genus $1$ then $\chi(\bar P_{D \slash C})=1$.
\end{cor}
\begin{proof}
Under the assumption, $\bar J_D$ admits a stratification in generalized Jacobians of partial normalizations of $D$. The stratification is indexed by the subset of the set $A$ of nodes of $D$ in the following way. There is a natural action of $\iota$ on $A$, so we can talk of $\iota$--invariant subsets of $A$. For every subset $B \subset A$, the stratum corresponding to the normalization $D_B$ of $D$ at the nodes of $B$ is isomorphic to 
\[
0 \to  (\mathbb C^*)^{\# A \setminus B} \to J_D \to J_{\wt D} \to 0.
\]
By the proposition above, such a stratum intersects the compactified Prym variety if and only if $B$ 
is $\iota$--invariant. If this is the case, then the induced stratum on the Prym is given by
\[
0 \to  (\mathbb C^*)^{(\# A \setminus B)\slash \iota} \to P_{D\slash C} \to P_{\wt D\slash \wt C}=\{pt\} .
\] 
Every stratum has trivial Euler number, except for the one corresponding to $B=A$, which is just one point.
\end{proof}

\section{Completion of the proof of Theorem \ref{mainthm}} \label{sec:mainthm}
By the discussion of Section \ref{secprelim} and Proposition \ref{chizero}, the only contribution to the Euler characteristic is due to compactified Pryms $\bar P_{D/C}$, where $C$ is an irreducible plane quintic of geometric genus $1$. By Lemma  \ref{only5A1} and Proposition \ref{verygeneralline}, such curves arise via projections from a general line $\ell$ on a $5$-nodal hyperplane section $Y=X\cap H$. By Theorem \ref{theoremK}, there are $176,904$ such hyperplanes. Finally, by Corollary \ref{chione}, the contribution of each such hyperplane is $1$. By Corollary \ref{cor_euler}, we conclude $\chi(Z)=\chi(\overline \calJ)=176,904\cdot 1$.
\qed


\bibliography{eulerOG10}
\end{document}